\documentclass[11pt]{amsart}
\usepackage{amsfonts}
\usepackage{amsthm}
\usepackage{xcolor}
\usepackage{hyperref}
\usepackage[margin=1.0in]{geometry}
\usepackage{graphicx}
\usepackage[numbers]{natbib}
\usepackage[final]{pdfpages}
\usepackage{amsmath}
\usepackage{subcaption} 
\usepackage[toc,page]{appendix}

\newtheorem{theorem}{Theorem}

\newtheorem{cor}{Corollary}

\theoremstyle{definition}

\begin{document}
\title[Tracking a Brownian Bridge]{Size and shape of tracked Brownian bridges}	
	\author{Abdulrahman Alsolami, James Burridge and Micha\l \ Gnacik}
	
	
	\begin{abstract}
		We investigate the typical sizes and shapes of sets of points obtained by irregularly tracking two-dimensional Brownian bridges. The tracking process consists of observing the path location at the arrival times of a non-homogeneous Poisson process on a finite time interval. The time varying intensity of this observation process is the \textit{tracking strategy}. By analysing the gyration tensor of tracked points we prove two theorems which relate the tracking strategy to the average gyration radius, and to the \textit{asphericity} -- a measure of how non-spherical the point set is. The act of \textit{tracking} may be interpreted either as a process of observation, or as process of depositing time decaying ``evidence'' such as scent, environmental disturbance, or disease particles.  We present  examples of different strategies, and explore by simulation the effects of varying the total number of tracking points.
	\end{abstract}
	\maketitle
	%
	%
	%
	%
	%
	
	
\section{Introduction}

Understanding the statistical properties of human and animal movement processes is of interest to ecologists \cite{sta11,cod08,gill88}, epidemiologists \cite{sch20,bal09,dan19}, criminologists \cite{tay14}, physicists and mathematicians \cite{bro06,gon08,gal18,mos07}, including those interested in the evolution of human culture and language \cite{liz11,for12,bur17}.  Advances in information and communication technologies have allowed automated collection of large numbers of human and animal trajectories \cite{zhe15,col06}, allowing real movement patterns to be studied in detail and compared to idealised mathematical models. Beyond academic study,  movement data has important practical applications, for example in controlling the spread of disease through contact tracing \cite{dan19}. Due to the growing availability and applications of tracking information, it is useful to possess a greater analytical understanding of the typical shape and size characteristics of trajectories which are observed, or otherwise emit information. In this paper we will consider how such characteristics change when observations occur in a time varying manner.

We are interested in average geometrical properties of Brownian bridges \cite{oks10}. These continuous random motions, conditioned to begin and end in the same location (the \textit{tether point}), are widely used as models of animal movement \cite{hor07}, for example when estimating the home ranges of snakes \cite{sil18}, or birds \cite{fis13}. The mathematical inspiration for our work comes from polymer physics, where the shapes of long chain molecules can be modelled as random walks \cite{rud86,rud87}. Polymer chains and random walks tend to be \textit{elongated} (they are not ``spherical'').  This means that the eigenvalues of their gyration tensor, which measures the distribution of the walk about its centre of mass, are not equal. In 1985 Rudnick and Gaspari showed \cite{rud86} that an \textit{average} measure of  deviation from equal eigenvalues, called  \textit{asphericity}, could be computed exactly for unrestricted walks, or estimated using series expansions for more complex objects.  In this paper we extend the original definition of asphericity by interpreting our bridge walks as movement trajectories for which location information is stored (or emitted) at a time varying rate. We say that such walks are \textit{tracked} \cite{gal18}.   The \textit{tracking strategy} $\mu$, is a probability density function which describes how the density of tracking data varies with time.  Observation times are modelled as a Poisson point process \cite{grim01} with intensity $c \mu(t)$, where $c$ measures the absolute intensity of observations. In contrast to the original definition of asphericity \cite{rud86}, we consider the distribution of tracked points about the tether point, rather than the centre of mass. In the limit $c \rightarrow \infty$ we are able to analytically compute, in terms of $\mu$, both the asphericity of the tracking data, and its radius of gyration (a simple measure of its overall spatial size), about the tether point. 


The function $\mu$ may be given other interpretations than tracking. For example it could be used to model scent decay rates \cite{bri02}, time variations in the density of communication events from mobile phones \cite{can08}, or memory (as in our previous work \cite{gn18}). Here we motivate our study with a possible application to the spread of disease. Consider a large population of foragers (human or animal), each of whom has a home location, from which they make foraging trips to gather resources. Figure \ref{fig:cont} shows an example of a single trip.
\begin{figure}
\centering
\includegraphics[scale=0.23]{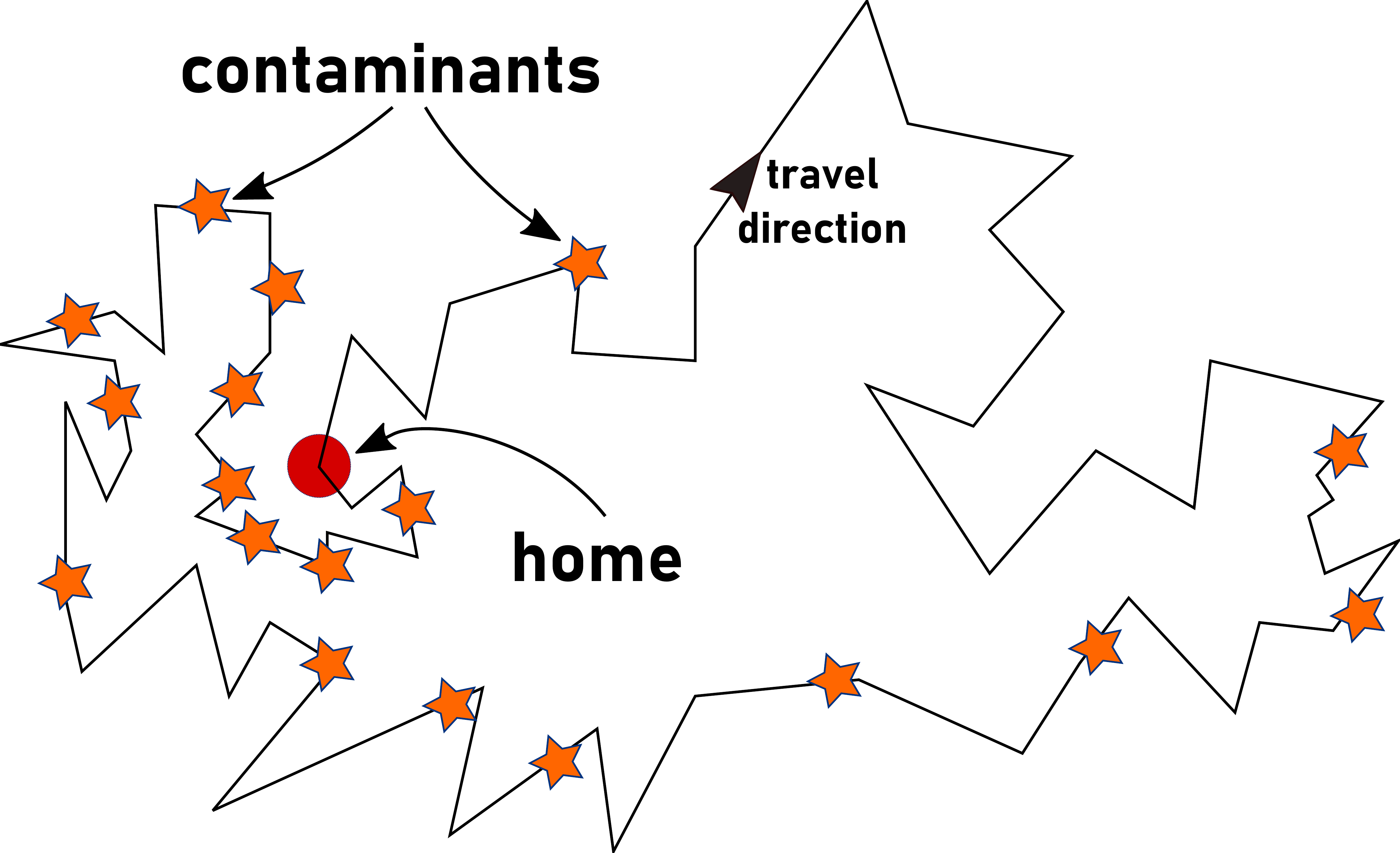}
\caption{Time decaying contaminant particles left in the environment by a foraging organism.}
\label{fig:contam}
\end{figure}
Suppose that some of the population are contaminated, for instance with an infectious disease which they leave traces of in the environment as they move. Diseases which can be spread in this manner include norovirus \cite{tow18},  acute respiratory illnesses such as influenza \cite{boo07} and diseases spread via rodent urine such as Leptospirosis (Weil's disease).  If these environmental traces decay with time then the density of active contaminants left behind by the forager will be greater in regions visited more recently. Different rates of contaminant decay, and times since the walk ended, correspond to different tracking strategies $\mu$. Size and shape characteristics of the contaminated region (its radius of gyration, and asphericity) may be computed, in terms of $\mu$, using the results presented in our paper. Predicting the sizes and shapes of regions contaminated by infected individuals may be useful for quantitatively understanding the spread of diseases via environmental ``fomite'' contamination \cite{boo07}. For example, combined with estimates of the number of infected individuals, such information would allow for estimates of the spatial distribution and density of contaminants.

	\section{Tracked Brownian Bridge}

	\subsection{Brownian bridge}
	
	Let $(B(t))_{t \geq 0}$ be a one-dimensional standard Brownian motion on $[0, \infty)$ (standard so that $B(0) = 0$ and $B(t)$ is normally distributed with zero mean and variance $t$, $t > 0$). A Brownian bridge $\widehat{B}$ that terminates at time $t=1$ \cite{oks10}, is a Brownian motion which is pinned, or tethered, at the origin at $t=0$ and $t=1$, namely,
$\widehat{B}(t) = (B(t) \ | \ B(1) = 0), \ t \in [0, 1]$.
 We note the following, equivalent definition: a Brownian bridge  $\widehat{B} = (\widehat{B}(t))_{t \in [0, 1]}$ is a Gaussian process such that 
$ \mathbb{E}[\widehat{B}(t)] = 0, \ 0 \leq t \leq 1, \ \
	\mbox{Cov}(\widehat{B}(t), \widehat{B}(s)) = \mathbb{E}[\widehat{B}(t)\widehat{B}(s)] = s(1-t), \ 0\leq s< t \leq 1.$
	
	\noindent 
	Given a standard Brownian motion $(B(t))_{t \geq 0}$ one may construct a Brownian bridge explicitly by the following formula 
	\begin{equation}
	\label{bb_eq} \widehat{B}(t) = B(t) -tB(1),\  t \in [0, 1].
	\end{equation}
	Given a pair $(B, \widehat{B})$ we will refer to $\widehat{B}$ as the corresponding Brownian bridge. Conversely, given a Brownian bridge $\widehat{B} = (\widehat{B}(t))_{t \in [0, 1]}$ one may construct a standard Brownian motion on $[0, 1]$  via 
	\begin{equation} \label{bm_eq}B(t) = \widehat{B}(t) + tZ,\end{equation}
	where $Z \sim \mathcal{N}(0, 1)$ is independent from $\widehat{B}$. Given a pair $(\widehat{B}, B)$ we will refer to $B$ as the corresponding Brownian motion. 
	
	The interval $[0, 1]$ may be extended to any finite time-interval $[0, T]$, $T>0$, so that the corresponding Brownian Bridge would be pinned at the origin at $t=0$ and $t=T$. Then instead of (\ref{bb_eq}) we would have $B(t) -\frac{t}{T}B(T)$.
Given a Brownian motion $B_1$ on $[0, 1]$, by scaling the time one may construct a Brownian motion $B_2$ on $[0, T]$ by setting 
$B_2(t) = \sqrt{T} B_1\left(\frac{t}{T}\right)$. For $t \in [0, T]$ we take the Brownian bridge
$\widehat{B_2}(t) = B_2(t) - \frac{t}{T}B_2(T)$
and note that  
\begin{equation}\label{eqn: bridges}\widehat{B_2}(t) = \sqrt{T}\widehat{B_1}\left(\frac{t}{T}\right),\end{equation} 
where $\widehat{B_1}(s) = B_1(s) - sB_1(1)$ for all $s \in [0, 1]$. Hence, further in this paper, without loss of generality we may consider Brownian bridges on $[0, 1]$ and use (\ref{eqn: bridges}) to obtain the relevant results for Brownian bridges on $[0, T]$.

	\subsection{Tracked Brownian bridge}
	
	Let $S$ be a non-homogeneous Poisson point process on $I:= [0, 1]$ with intensity function $\lambda(t)$. Equivalently,  $S$ is the set of arrival times of a corresponding non-homogeneous Poisson counting process on $I$.  In particular, we consider $S$ with intensity of the form 
	$\lambda(t) = c\mu(t),$ where $c>0$ and $\mu$ is a probability density function of a probability distribution supported on $[0, 1]$.	Let $\widetilde{S} := S \cup \{0, 1\}$.	Given a Brownian bridge $\widehat{B}$ we will call the process $(\widehat{B}(t))_{t \in \widetilde{S}}$ a \emph{one-dimensional tracked Brownian Bridge} with tracking strategy kernel $\mu$ and rate of intensity $c$. Figure \ref{fig:exp_bridge} shows an example of a tracked bridge generated by an exponential strategy. We focus on \emph{two dimensional tracked Brownian bridges}, that is, $(\widehat{X}(t), \widehat{Y}(t))_{t \in \widetilde{S}}$, where $\widehat{X}$ and $\widehat{Y}$ are two independent Brownian bridges, with tracking strategy kernel $\mu$ and the rate of intensity $c$, as $c \to \infty$. 
	
	\begin{figure}
		\centering
		\includegraphics[scale=0.4]{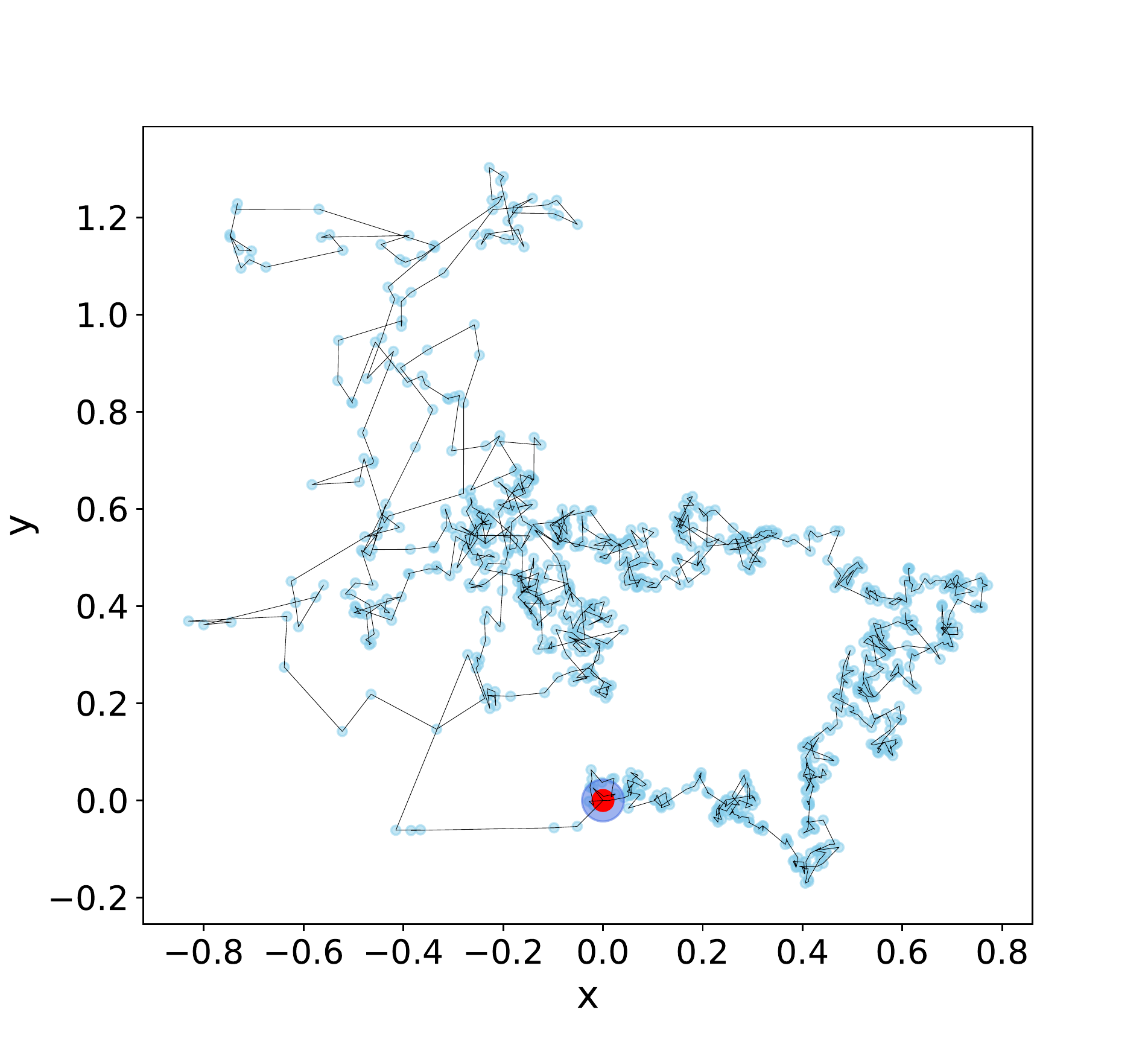}
		\caption{Observed locations of a tracked bridge generated using exponential kernel with $\lambda=5$. Intensity $c=1000$. Red dot at (0,0) shows the start / end point of the walk.}
		\label{fig:exp_bridge}
	\end{figure}

		

	\subsection{Gyration tensor, radius of gyration and asphericity measure}
	To characterise the shapes of tracked Brownian bridges we use \emph{gyration tensors}, which were originally used by Rudnick and Gaspari \cite{rud86} to define a measure of asphericity of random walks.  In our previous paper \cite{gn18} we worked with \emph{egocentric asphericity} which differs from the standard definition of asphericity in that the moments which form the elements of the gyration tensor are taken about the current location of the walker, rather than the centre of mass of the walk. In the current paper moments are taken about the tether point $(0,0)$ of the bridge, thought of as the home location of the walker. The gyration tensor captures the spatial distribution of the tracked bridge about the tether point.	Let $(\widehat{X}(t), \widehat{Y}(t))_{t \in \widetilde{S}}$ be a tracked Brownian Bridge with intensity rate $c$ and tracking strategy kernel $\mu$;
	we can view it as a (random) set $L_c$ of all observed locations of the walker, that is, $L_{c} = \{ (\widehat{X}(t), \widehat{Y}(t)) \colon t \in  \widetilde{S}\}$. The corresponding gyration tensor (about the tether point) is defined via
	\begin{equation}
	T_{c} = \left[\begin{smallmatrix} T_{11} & T_{12}  \\ T_{12} & T_{22 }\end{smallmatrix} \right] = \left[\begin{smallmatrix} \frac{1}{2+|S|}\sum_{t \in S}\widehat{X}(t)^2 & \frac{1}{2+|S|}\sum_{t \in S}\widehat{X}(t)\widehat{Y}(t) \\ \frac{1}{2+|S|}\sum_{t \in S}\widehat{X}(t)\widehat{Y}(t) & \frac{1}{2+|S|}\sum_{t \in S}\widehat{Y}(t)^2\end{smallmatrix} \right].
	\end{equation} 
	The \textit{tether point asphericity} of $L_{c}$ is defined as
	\begin{equation}
	A_c = \frac{\mathbb{E}[(\lambda_1 - \lambda_2)^2]}{\mathbb{E}([\lambda_1 + \lambda_2)^2]} 
	= 1 - 4 \frac{\alpha_c}{\beta_c},
	\end{equation}
	where $\lambda_1, \lambda_2$ are the eigenvalues of $T_c$ and
	\begin{equation}
	\alpha_c = \mathbb{E}[T_{11}T_{22}] - \mathbb{E}[T_{12}^2], \qquad \beta_c = \mathbb{E}[(T_{11}+T_{22})^2].
	\end{equation}
	The trace of $T_c$, that is, $r_c^2 :=\mbox{tr}(T_c) = T_{11}+T_{22}$ is the \emph{radius of gyration}, and determines the size of the walk. Similarly to \cite[formula (4) p. 7017]{sci94} we are taking a square of the usual radius of gyration. A geometrical understanding of tether point asphericity is as follows. Given a realization of the random set $L_c$, the eigenvalues of its gyration tensor give the mean squared deviations of $L_c$ along the principle axes. If one of the eigenvalues is significantly larger than the other, then the spatial distribution of observed locations around the tether point is typically elongated along this direction in space. If, on average, a tracking strategy produces points sets $L_c$ which are very elongated (one eigenvalue much larger than the other), then the tether point asphericity approaches one, which is its maximum value. If the strategy usually generates point sets with approximately equal eigenvalues, then the tether point asphericity of the collection of random sets $L_c$ will be close to zero, meaning that they are \textit{typically} not elongated. Because the elements of our gyration tensor are mean square deviations about the tether point and not the centre of mass of the tracked points, a set of points which would be perfectly spherical by the centre of mass definition are not spherical by ours (see \cite{gn18} for some extreme examples). Despite this, our definition remains a measure of typical elongation (see Figure \ref{fig:cont}) and due to its connection to the original definition we keep the name. We also note that the expected difference between the eigenvalues of the tether point gyration tensor characterise the distribution of tracked points about the home location using a single number. Henceforth we will refer to tether point asphericity simply as asphericity, in the interests of brevity.

		As $c$ becomes infinitely large, the elements of $T_c$  maybe approximated in distribution by the corresponding elements of $T$ defined as follows
		\begin{equation}
		T =  \left[\begin{smallmatrix} \int_0^1 {\widehat{X}(t)}^2 \mu(t) dt &  \int_0^1 {\widehat{X}(t)} {\widehat{Y}(t)}\mu(t) dt  \\ \int_0^1 {\widehat{X}(t)} {\widehat{Y}(t)}\mu(t) dt  & \int_0^1 {\widehat{Y}(t)}^2 \mu(t) dt \end{smallmatrix} \right].
		\end{equation} 
		This observation arises from Campbell's theorem \cite{str10, king02} on the characteristic function of a time-function summed over a point process (see also Appendix A in \cite{gn18}) and	our numerical evidence from the simulation section \ref{sim}. 	The moments of Poisson sums/integrals with stochastic summands/integrands were investigated in \cite{pri12}. Due to the above approximation we have that as $c$ becomes large
		$A_c \approx A = 1 - 4 \frac{\alpha}{\beta}$, where 
		\begin{align} \label{eqn: alpha}
		\alpha =& \left(\mathbb{E}\left[  \int_0^1 {\widehat{X}(t)}^2 \mu(t) dt\right] \right)^2 - \mathbb{E}\left[\left( \int_0^1 {\widehat{X}(t)} {\widehat{Y}(t)}\mu(t) dt \right)^2 \right]\\ \label{eqn: beta}
		\beta =& 2\mathbb{E}\left[ \left( \int_0^1 {\widehat{X}(t)}^2 \mu(t) dt\right)^2\right] + 2\left(\mathbb{E}\left[  \int_0^1 {\widehat{X}(t)}^2 \mu(t) dt  \right]\right)^2.
		\end{align} 
		Moreover, the average value of the radius of gyration $r_c^2$ can be approximated by
	\begin{equation}
		r^2=2\left(\mathbb{E}\left[  \int_0^1 {\widehat{X}(t)}^2 \mu(t) dt  \right]\right).
		\end{equation}
	
Note that matrix $T$ is almost surely (meaning with probability one) positive-definite since it is symmetric and all its pivots are  almost surely positive. Hence, $T$ is diagonalisable and so there exists a rotation matrix $R(\theta)$ with some $\theta \in [0, 2\pi)$ such that
	\begin{equation}\label{eqn: rot}(R(\theta))^T T R(\theta) = \left[\begin{array}{cc}\lambda_1 & 0 \\ 0 & \lambda_2 \end{array} \right],\end{equation}
	where $\lambda_1>0$ and $\lambda_2>0$ are eigenvalues of $T$. In particular, we choose $R(\theta)$ so that $\lambda_1 \geq \lambda_2$. 
	An ellipse with the same asphericity as the tensor $T$ is given by
\begin{equation}
v^TT^{-1}v = \kappa^2,
\label{eqn:ellip}
\end{equation}
	where $v \in \mathbb{R}^2$, $\kappa >0$.
	Different choices of $\kappa$ do not influence the asphericity, this parameter simply rescales the overall size of the ellipse. See Figure \ref{fig:loc_sets} for some examples of tracked walks and the ellipses determined by their gyration tensors.

	\section{Main results}
	Let $(\widehat{X}(t), \widehat{Y}(t))_{t \in \widetilde{S}}$ be a two-dimensional tracked Brownian bridge with tracking strategy kernel $\mu$ and the rate of intensity $c>0$. We introduce the following notation $$M_0(t) = \int_0^t \mu(s)ds, \ M_{1}(t) = \int_0^tM_0(s)ds \mbox{ and } M_2(t)=\int_0^t M_1(s)ds.$$ 
		All calculations for the below proofs are provided in the supplementary material. 
	\begin{theorem} \label{thm1}
		The average radius of gyration of  $(\widehat{X}(t), \widehat{Y}(t))_{t \in \widetilde{S}}$, as $ c \to \infty$ is
		\begin{equation}r^2 = 2M_{1}(1)-4M_{2}(1).\end{equation}
	\end{theorem}
	\begin{proof}
	   Recall that $r^2 = 2\mathbb{E}\left[  \int_0^1 {\widehat{X}(t)}^2 \mu(t) dt\right]$. Let $X(t) = \widehat{X}(t) + t\! \! \!\! \! \!\overbrace{X(1)}^{=:Z \sim \mathcal{N}(0, 1)}$ be the corresponding Brownian motion then by expanding the bracket ${\widehat{X}(t)}^2 = (X(t) - tX(1))^2$ and applying \cite[Lemma B.1]{gn18} we obtain 
	   \begin{align*}
	   	\int_0^1\! {{\widehat{X}}(t)}^2\mu(t) dt =& 2X^2(1)M_{2}(1) -\!M_{1}(1) -2 \!\int_0^1\! M_{0}(t)X(t)dX(t)-\!2X(1)\!\int_0^1 \! \left(M_{1}(t) - t M_{0}(t)\right) \!dX(t).
	   \end{align*}
	  We then take the expectation and apply Ito's
	  isometry to the last term so that the result holds.
	\end{proof}
	The result from Theorem \ref{thm1} can be generalised to Brownian bridges that terminate at an arbitrary time $T>0$ and have volatility $\sigma>0$. We capture that in the following corollary:
	\begin{cor}
	Let $T>0$, $\sigma>0$ and let  $\widehat{X_0}$ and $\widehat{Y_0}$ be Brownian bridges that terminate at $T$ and let $S$ be a Poisson point process on $[0, T]$ with intensity function $\lambda(t) = c\nu(t)$ for some $c>0$ and an integrable function $\nu$ supported on $[0, T]$. The average radius of gyration of $(\sigma \widehat{X_0}(t), \sigma \widehat{Y_0}(t))_{t \in S \cup \{0, 1\} }$, as $ c \to \infty$ is
	$$r^2 = 2 \sigma^2 \left( N_1(T) - \frac{2}{T}N_2(T)\right),$$
	where $N_0(t) = \int_0^t \nu(s)ds$, $N_1(t) = \int_0^t N_0(s)ds$, \ $N_2(t) = \int_0^t N_1(s)ds$.
	\end{cor}
	\begin{proof}
	By virtue of (\ref{eqn: bridges}) there are Brownian bridges $\widehat{X}$ and $\widehat{Y}$ that terminate at $1$
	so that 
	\begin{align*}
	 \widehat{X_0}(t) = \sqrt{T}\widehat{X}\left(\frac{t}{T}\right), \  \widehat{Y_0}(t) = \sqrt{T}\widehat{Y}\left(\frac{t}{T}\right).
	\end{align*}
	Now we are ready to find the gyration tensor: 
	\begin{align*}
	    r^2 =& 2\mathbb{E}\left[  \int_0^T {\left(\sigma\widehat{X_1}(t)\right)}^2 \nu(t) dt\right] = 2T^2\sigma^2\mathbb{E}\left[  \int_0^1 {\widehat{X}(s)}^2 \mu(s) ds\right],
	\end{align*}
	where $\mu(s) = \nu(sT)$. Thus the result follows from the preceding theorem. 
	\end{proof}

	In next result  we specify the asphericity of the tracked Brownian bridge as $c$ becomes infinitely large.

	\begin{theorem}\label{thm: main_thm}
		The asphericity of $(\widehat{X}(t), \widehat{Y}(t))_{t \in \widetilde{S}}$, as $ c \to \infty$ is
		$A = 1 - 4\frac{\alpha}{\beta},$ where
		\begin{align}\label{eqn: alpha1}
		\alpha =& 4M_{0}(1)(M_{1}(1)-2M_{2}(1))-M_{1}^2(1)+4M_{1}(1)M_{2}(1)\\ \nonumber 
		&+4 \int_0^1 t\mu(t)(2M_{2}(t)- tM_{1}(t))dt-2 \int_0^1 tM_{0}^2(t)dt -	2\int_0^1 (M_{1}(t)-tM_{0}(t))^2 dt
		\end{align}
		and
		\begin{align}\label{eqn: beta1}
		\beta =& -4(M_{1}^2(1)-4M_{1}(1)M_{2}(1) + 8M_{2}^2(1))
		\\ \nonumber 
		&+8 \int_0^1 (1-t)tM_{0}^2(t)dt + 16 \int_0^1 (1-2t)M_{0}(t)(2M_{2}(t)-t M_{1}(t))dt\\ \nonumber 
		&+8 \int_0^1 M_{1}(t)((1-4t)M_{1}(t)+8M_{2}(t))dt.
		\end{align}
	\end{theorem}
	\begin{proof}
	   	Recall that $\alpha$ and $\beta$ are given by (\ref{eqn: alpha}), (\ref{eqn: beta}) so that we need to calculate the following expectations
		\begin{align} \label{exp}
		\mathbb{E}\left[  \int_0^1 {\widehat{X}(t)}^2 \mu(t) dt\right], \
		\mathbb{E}\left[ \left( \int_0^1 {\widehat{X}(t)}^2 \mu(t) dt\right)^2\right] \mbox{ and } \mathbb{E}\left[\left( \int_0^1 {\widehat{X}(t)} {\widehat{Y}(t)}\mu(t) dt \right)^2 \right].
		\end{align} 
		In the preceding theorem we have already found the first expectation term in (\ref{exp}). For the last two terms in (\ref{exp}) one can show that by employing the fact that 
		$\widehat{X}(t) = X(t) - t\! \! \!\! \! \!\overbrace{X(1)}^{=:Z_1 \sim \mathcal{N}(0, 1)}$, 	$\widehat{Y}(t) = Y(t) - t \! \! \!\! \! \!\overbrace{Y(1)}^{=:Z_2 \sim \mathcal{N}(0, 1)}$, where $X(t) = \widehat{X}(t) + tX(1)$ and $Y(t) = \widehat{Y}(t) + tY(1)$ are two independent Brownian motions, and using the results from our earlier work \cite[Lemma B.2 and Lemma B.5]{gn18}, we have
		\begin{align*}
			&\mathbb{E}\left[\left(\int_0^1 {\widehat{X}}^2(t)\mu(t)dt\right)^2 \right] =
			-3M_{1}^2(1) + 12 M_{1}(1)M_{2}(1) - 20M_{2}^2(1)
			+4 \int_0^1 (1-t)tM_{0}^2(t)dt \\
			&+ 8 \int_0^1 (1-2t)M_{0}(t)(2M_{2}(t)-t M_{1}(t))dt+4 \int_0^1 M_{1}(t)((1-4t)M_{1}(t)+8M_{2}(t))dt
			\end{align*}	
			and 
			\begin{align*}
	&\mathbb{E}\left[\left(\int_0^1 \widehat{X}(t)\widehat{Y}(t) \mu(t) dt\right)^2 \right]= 2M_{1}^2(1)+4M_{2}^2(1)-8M_{1}(1)M_{2}(1)+4M_{0}(1)(2M_{2}(1)-M_{1}(1))\\&+2 \int_0^1 tM_{0}^2(t)dt
	+	2\int_0^1 (M_{1}(t)-tM_{0}(t))^2 dt -4 \int_0^1 t\mu(t)(2M_{2}(t)- tM_{1}(t))dt.
	\end{align*}
	Hence, after all the expectations in (\ref{exp}) have been calculated then the main result holds after algebraic simplification of formulae (\ref{eqn: alpha}) and (\ref{eqn: beta}).
	\end{proof}

	\section{Examples}\label{sec: exa}
	
	
In this section we present examples of different tracking strategies $\mu$ and their corresponding asphericities $A$ and radii of gyration $r^2$. 
\subsubsection*{Uniform tracking strategy}
Let $0<s<1$, we consider
$\mu(t) = \frac{1}{s} \mathbf{1}_{[0,s]}(t)$ 
then the radius of gyration is given by 
$r_{s}^2 = s -\frac{2}{3}s^2.$
It attains its maximum value $\frac{3}{8}$ at $s = \frac{3}{4}$, and also 
$\lim_{s \to 1^{-}} r_{s}^2 = \frac{1}{3}.$
The asphericity is given by
$ A(s) =  1 - \frac{15-12s}{40s^2-108s+75}.$
In particular, $A(s)$ is decreasing on $[0, 1]$ and $\lim_{s \to 0^+} A(s) = \frac{4}{5}$, 
$\lim_{s \to 1^-} A(s) = \frac{4}{7} \approx 0.5714.$

\subsubsection*{Exponential tracking strategy}

Let $\lambda>0$ and
$\mu(t) = \frac{\lambda e^{-\lambda t}}{1-e^{-\lambda}}$ be supported on $[0, 1]$, then radius of gyration is given by
$r_{\lambda}^2 = \frac{2 \left(e^{\lambda } (\lambda -2)+\lambda +2\right)}{\left(e^{\lambda }-1\right)
	\lambda ^2}$. In particular, $r_{\lambda}^2$ is decreasing on $(0, \infty)$ with $\lim_{\lambda \to 0^+}r_{\lambda}^2 = \frac{1}{3}$ and $\lim_{\lambda \to \infty} r_{\lambda}^2 =0$. The asphericity is given by \begin{equation}
A(\lambda)=\frac{2 ((\lambda^2 + 8) \cosh(\lambda) - 5 \lambda \sinh(\lambda) - 8)}{2 (\lambda^2 - 8) + (3 \lambda^2 + 16) \cosh(\lambda) - 13 \lambda \sinh(\lambda)}.
\label{eqn:AExp}
\end{equation}
In particular, $A(\lambda)$ is increasing on $(0, \infty)$ and $\lim_{\lambda \to 0^+}A(\lambda) = \frac{4}{7} \approx 0.5714$, 
$\lim_{\lambda \to \infty }A(\lambda) = \frac{2}{3}.$
\subsubsection*{Triangular tracking strategy}
Let $a \in (0, 1)$ and
$$\mu(t) =\left\{ \begin{array}{ll} \frac{2}{a}t & \mbox{if } 0 \leq t \leq a \\
\frac{2}{a-1}t + \frac{2}{1-a}& \mbox{if } a <t<1\\
0 & \mbox{otherwise.} \end{array}\right. $$
The radius of gyration is
$
r_a^2 =\frac{1}{3} \left(-a^2+a+1\right)
$, it has its maximum value $\frac{5}{12}$ at $a = 0.5$. The asphericity is given by
$$
A(a) = \frac{15 a^4-30 a^3+9 a^2+6 a+3}{11 a^4-22 a^3+a^2+10 a+5},
$$
with the maximum value  $A(a_0) = \frac{87}{131}\approx 0.664122$ at $a_0=\frac{1}{2}$ and
$\lim_{a \to 0^+} A(a) = \lim_{a \to 1^-} A(a) = \frac{3}{5}.$
\subsubsection*{Inverted triangular tracking strategy} Let $a \in (0, 1)$ and 
\begin{align*}
\mu(t) =& \begin{cases}-\frac{2}{a}t+2 & \mbox{ if } 0\leq t \leq a \\
\frac{2}{1-a} t + \frac{2a}{a-1} & \mbox{ if } a \leq t \leq 1 \\
0 & \mbox{ otherwise }
\end{cases}
\end{align*}
then radius of gyration is 
$r_a^2 =\frac{1}{3} \left(a^2-a+1\right)$
and so
$
\lim_{a \to 0^{+}} r_a^2 =
\lim_{a \to 1^{-}} r_a^2 = \frac{1}{3}.
$
Also, observe that 
$r_a^2$ has its minimum value at $a=\frac{1}{2}$ with value $\frac{1}{4}$. The asphericity is given by $$A(a)=\frac{37 a^4-74 a^3+11 a^2+26 a-15}{a^4-2 a^3-47 a^2+48 a-25}.$$
Its minimum is
at $a_0=\frac{1}{2}$ with value	$A(a_0) = \frac{11}{23} \approx 0.478261$, maximum at  $a_1 = \frac{1}{2} \pm \frac{\sqrt{21}}{10}$ with value $A(a_1) = \frac{362}{601} \approx 0.602329$. It is increasing on intervals $\left(0, \frac{1}{2}-\frac{\sqrt{21}}{10}\right)$, $\left(\frac{1}{2}, \frac{1}{2}+\frac{\sqrt{21}}{10}\right)$ and decreasing on $\left(\frac{1}{2}-\frac{\sqrt{21}}{10}, \frac{1}{2}\right)$, $\left(\frac{1}{2}+\frac{\sqrt{21}}{10}, 1\right)$.

\subsubsection*{U-shaped tracking strategy}
Let $k \in \mathbb{N}$ and
$\mu(t) = (2k+1)(2t-1)^{2k}$ then
the radius of gyration is
$ r_k^2 = \frac{1}{3+2k}.$
As $r_k^2$ is decreasing we have that its maximal value is $r_1^2 = \frac{1}{5}$ and also
$\lim_{k \to \infty} r_k^2 = 0.$
The asphericity is given by \begin{equation}A(k)=\frac{4 \left(2 k^2+4 k+3\right)}{20 k^2+40 k+21}. \label{eqn:AU}\end{equation}
It is decreasing with maximum
$A(1) = \frac{4}{9}$
and $\lim_{k \to \infty} A(k) =  \frac{2}{5}.$

		\begin{figure}
		\centering
		\includegraphics[scale=0.25]{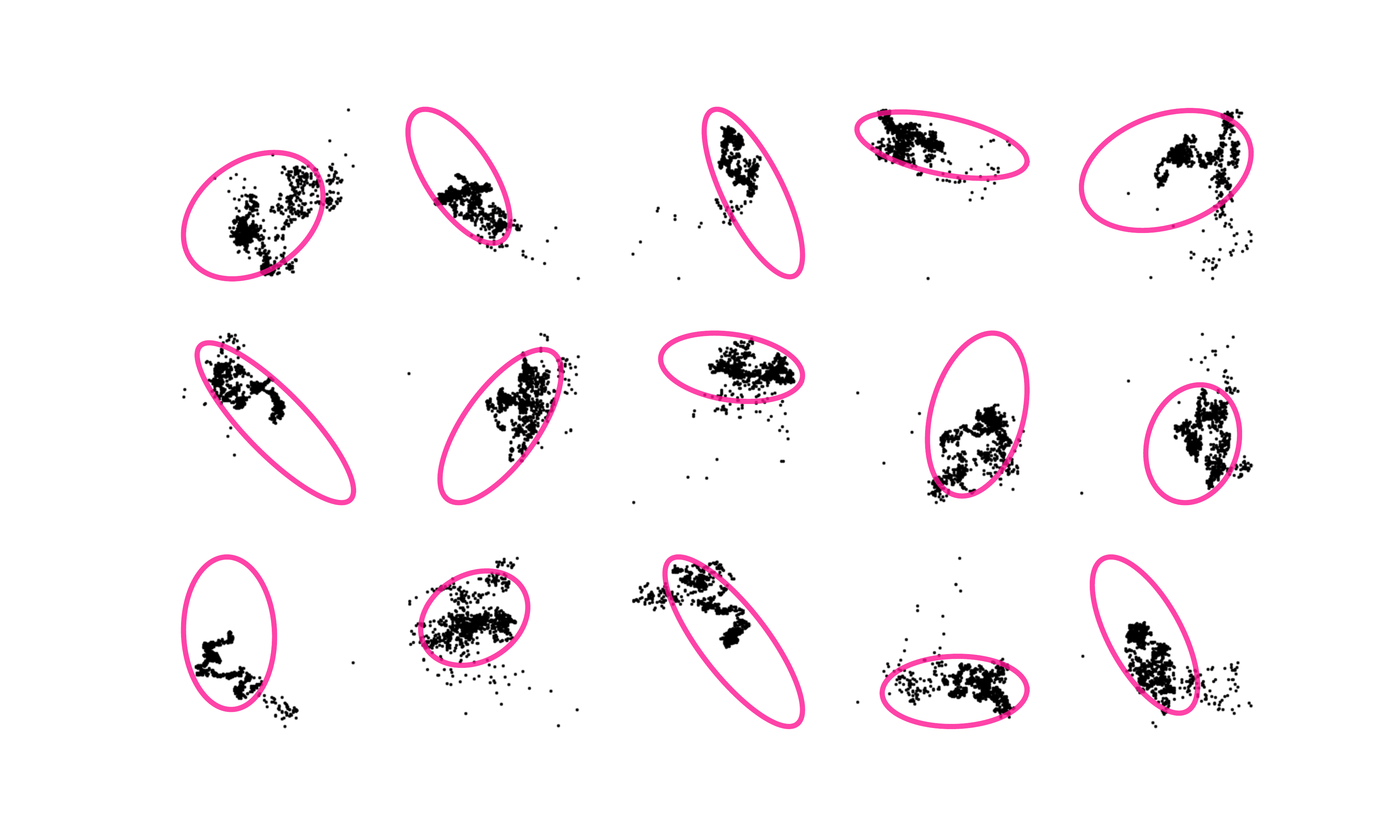}
		\caption{Location sets $L_c$ for exponentially tracked bridges with $\lambda=20$ and $c=1000$. Pink ellipses show solutions to equation (\ref{eqn:ellip}) with $\kappa=1$.}
		\label{fig:loc_sets}
	\end{figure}

\begin{figure}
  \begin{subfigure}[b]{0.4\textwidth}
    \includegraphics[width=\textwidth]{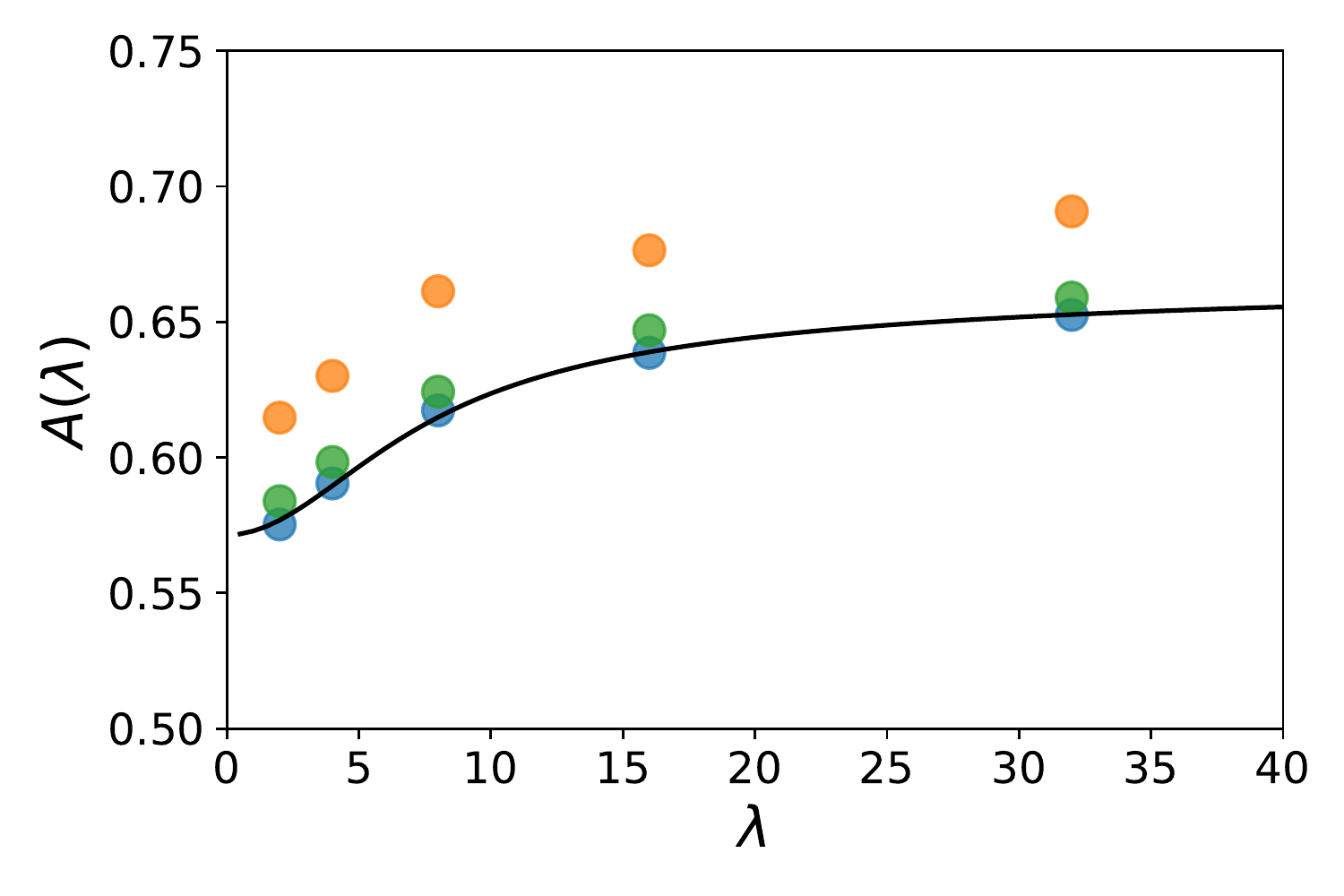}

  \end{subfigure}
  \begin{subfigure}[b]{0.4\textwidth}
    \includegraphics[width=\textwidth]{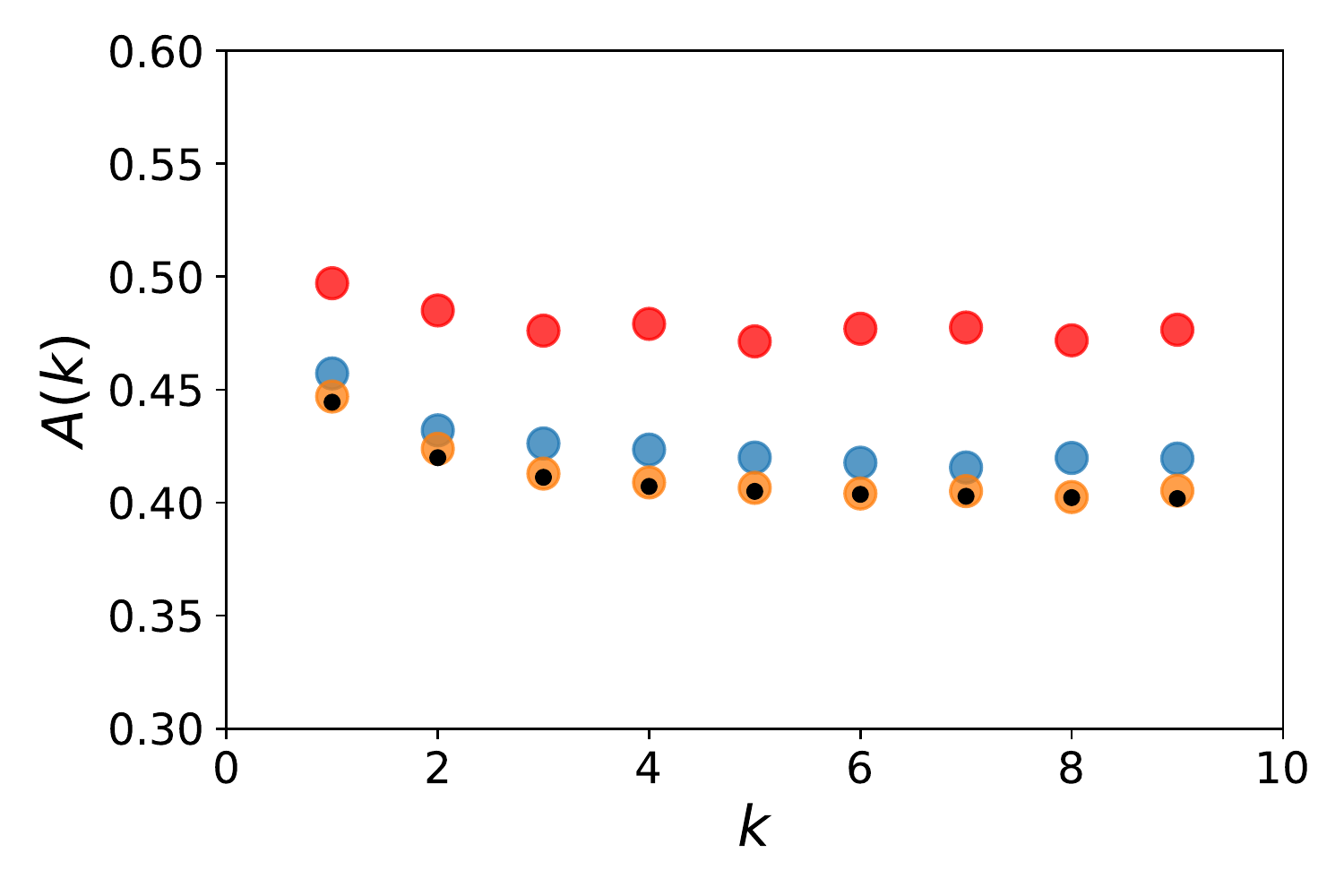}

  \end{subfigure}
   \caption{LHS figure: Asphericities of the tracked bridge with exponential kernel. Blue, green and orange markers correspond to intensities $c=1000,100,20$ respectively. Black line gives analytical form for $A(\lambda)$ given in equation (\ref{eqn:AExp}), which holds in the limit $c \rightarrow \infty$. Estimates computed using $N=10^4$ tracked bridges. $\\$ RHS figure: Asphericities of the tracked bridge bridge with U-shaped kernel. Orange, blue and red markers correspond to intensities $c=1000,100,20$ respectively. Black dots show analytical values of $A(k)$ given by equation (\ref{eqn:AU}), which hold in the limit $c \rightarrow \infty$. Estimates computed using $N=10^4$ tracked bridge bridges.  }
   	\label{fig:U_asp}
   	\label{fig:exp_asp}
\end{figure}

	\section{Simulations}\label{sim}

	We now test our analytical predictions against simulation. The construction of a tracked bridge consists of two steps. We first generate the set $ \widetilde{S} := S \cup \{0,1\}$, where $S$ are the arrival times of an inhomogeneous Poisson point process on $[0,1]$ with intensity $t \mapsto c \mu(t)$. We then generate the locations of a bridge at the set of times $\widetilde{S}$. To generate $\widetilde{S}$, we note that at time $s$, the cumulative distribution of the next arrival time, $T$, is
	\begin{align}
P_s(t):=	\mathbb{P}(T \leq t \ | \ t>s) 
	= \begin{cases}
	1 - \exp\left( -c \int_s^t \mu(u) du \right) &\text{ if } t \geq s \\
	0 & \text{ otherwise. }
	\end{cases}
	\end{align}
	Notice that $P_s(1) \leq 1$ because it is possible that there are no arrivals after time $s$. We now define the inverse of $P_s(t)$, modified so that its domain is $[0,1]$ 
	\begin{equation}
	P_s^{-1}(u) = \begin{cases}
	M_0^{-1} \left(M_0(s) - \frac{\ln(1-u)}{c} \right) & \text{ if } u\leq P_s(1) \\
	1 & \text{ if } u > P_s(1)
	\end{cases}
	\label{eqn:invc}
	\end{equation}
	where $M_0$ is the cumulative of the density $\mu$, earlier defined. The case $u<P_s(1)$ in (\ref{eqn:invc}) ensures that $P^{-1}: [0,1] \rightarrow [0,1]$. We now generate the set $\widetilde{S}$ beginning with time $T_0=0$ and generating the sequence $T_1,T_2, \ldots$ using
	$T_{k+1} = P^{-1}_{T_k}(U_k)$
	where $U_{k} \sim U[0,1]$ are independent uniform random variables. The sequence terminates when $T_k=1$. Having generated $\widetilde{S}$ we compute the locations of two independent Brownian motions at the times in $\widetilde{S}$ using $B_i(T_{k+1}) = B_i(T_k) + Z_{k+1}$ where $i \in \{1,2\}$ and $Z_{k+1} \sim \mathcal{N}(0,T_{k+1}-T_k)
$. The tracked bridge locations are given by
$\widehat{B}_i(T_k) = B_i(T_k) -T_k B_i(1)$. Figure \ref{fig:exp_bridge} 
	shows a tracked bridge generated by this method using exponential kernel. To estimate asphericities we generate a large number, $N$, of tracked bridges and compute the corresponding set of gyration tensors $T(1),T(2), \ldots , T(N)$. We then compute the estimators
	\begin{align}
	\widehat{\alpha}_c = \frac{1}{N} \sum_{n=1}^N \left(T_{11}(n) T_{22}(n) - T_{12}^2(n) \right), \ \ 
	\widehat{\beta}_c = \frac{1}{N} \sum_{n=1}^N \left(T_{11}(n) + T_{22}(n)\right)^2. 
	\end{align}
	Our asphericity estimator is then $\widehat{A} = 1 - 4\widehat{\alpha}_c/\widehat{\beta}_c$.

	To illustrate the correspondence between analytical and simulated asphericities and gyration radii we consider in detail the exponential and U-shaped kernels. In Figure \ref{fig:exp_asp} we have estimated the asphericity of exponentially tracked bridges for a series of values of the kernel parameter, $\lambda$, using three intensity levels, $c \in \{20,100,1000\}$. For finite intensity our analytical results underestimate the true asphericity, with the discrepancy rapidly reducing as $c$ approaches 100. A heuristic argument for this effect is that the subset of points along the continuous bridge which define the tracked bridge generate a structure whose convex hull lies inside that of the full bridge. The distance between the inner and outer hulls is a decreasing function of intensity. If we take an elongated hull, and move all of its boundaries inward by a fixed distance, then we produce a shape with greater asphericity. Hence, bridges tracked at lower intensity have larger asphericity. We also note from Figure \ref{fig:exp_asp} (LHS) that as $\lambda \rightarrow \infty$ the asphericity tends to $\tfrac{2}{3}$ which is identical to the asphericity of a standard, untethered random walk tracked using an exponential strategy \cite{gn18}  (note that in \cite{gn18}, $\mu$ was viewed as the memory kernel of a forager). This is a result of the fact that in the earliest part of its motion a bridge behaves like an untethered random walk, and for large $\lambda$ only the earliest parts of the walk are tracked. Figure \ref{fig:U_asp} (RHS) shows asphericity estimates using tracked bridges generated by the U-shaped kernel. Here we see that the effect of intensity is qualitatively similar to the exponential kernel case, with low intensity tracked bridges being on average more aspherical. For a U-shaped kernel the tracked points occur mainly at the beginning and end of the walk, so we are in effect computing the asphericity of two random walk paths, both starting from the same location. The superposition of two walks produces a structure which is less elongated than a single walk. These two walks are not independent, but as $k \rightarrow \infty$ the two (increasingly short) parts of the walk which are tracked behave progressively more like independent walks, further reducing the asphericity.  Predicted asphericities for all other tracking strategies have been verified by simulation for large $c$.   
	
	In Figure \ref{fig:exp_rad} we compare analytical values of the  gyration radius of an exponentially tracked walk to simulated values for finite observation intensities. When the total number of observations is small we underestimate the radius by a small fraction ($\approx 10 \%$ for the lowest intensity). The simulated radius converges to our analytical result with increasing $c$. For a given intensity, the magnitude of the discrepancy is approximately proportional to the gyration radius, so the relative error appears to be independent of $\lambda$. Similar radius underestimation effects are observed for the U-shaped kernel. 
	
	\begin{figure}
		\centering
		\includegraphics[scale=0.5]{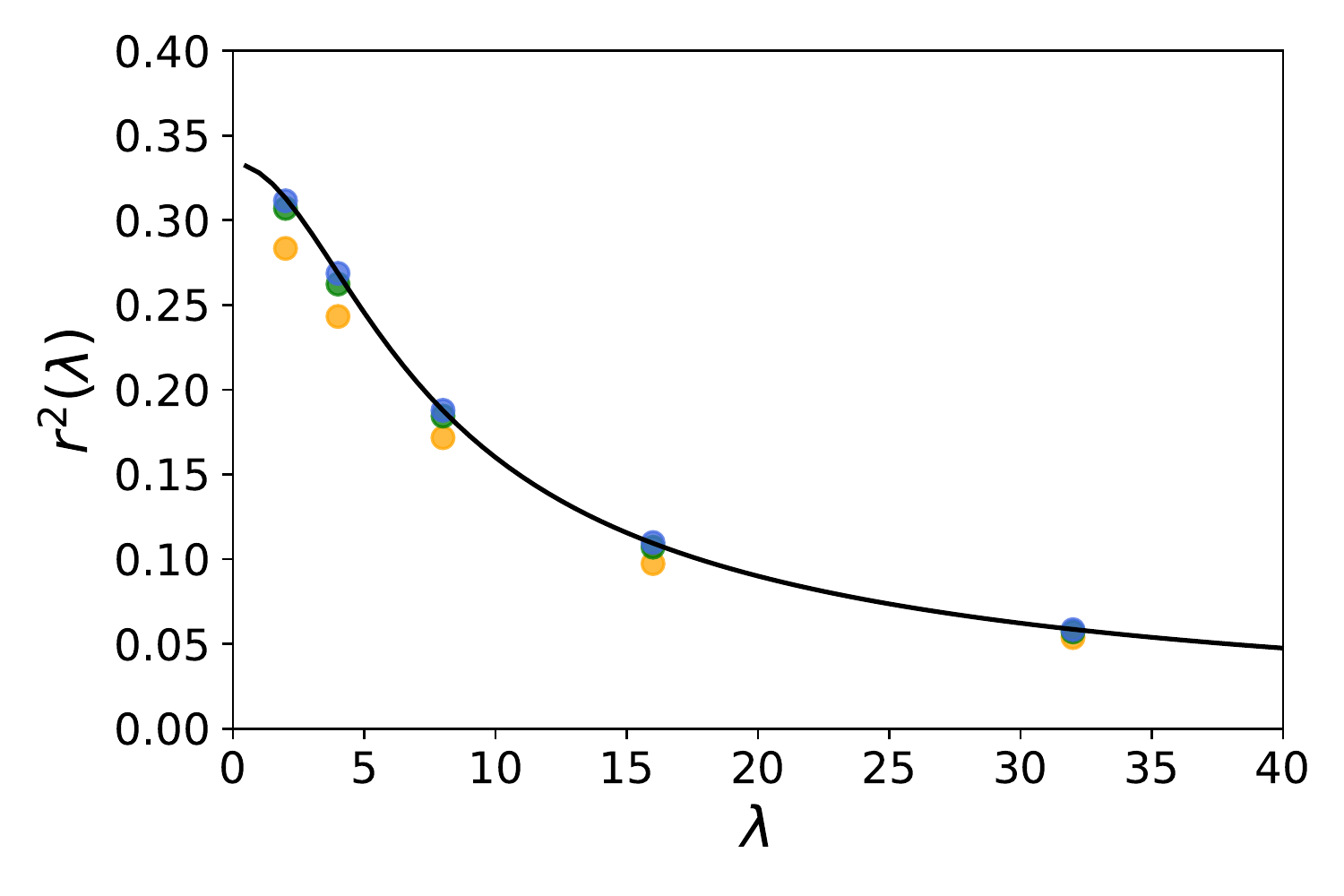}
		\caption{Radii of gyration of the tracked bridge with exponential kernel. Blue, green and orange markers correspond to intensities $c=1000,100,20$ respectively. Black line gives analytical form for $r^2(\lambda)$ given section \ref{sec: exa} Examples: Exponential tracking strategy, which holds in the limit $c \rightarrow \infty$. Estimates computed using $N=10^4$ tracked bridges. }
		\label{fig:exp_rad}
	\end{figure}

	\begin{figure}
		\centering
		\includegraphics[width=\linewidth]{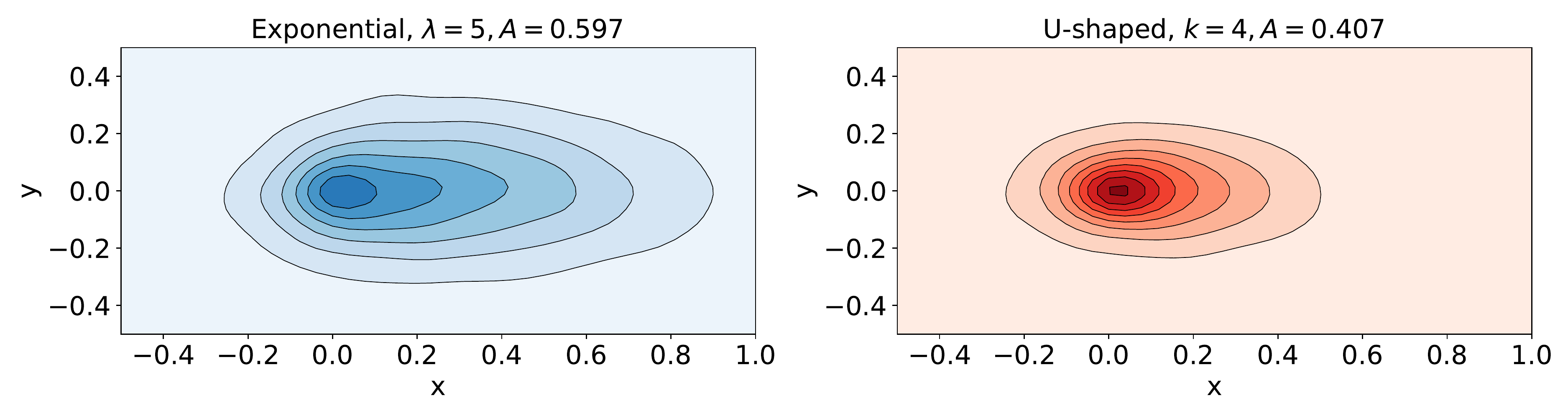}
		\caption{Probability density estimates of tracked bridges with $c=1000$, rotated so that their principle axes are horizontal and their centres of mass are to the left of the origin. Gaussian kernel density estimation bandwidth $=0.05$. In exponential (blue) plot contours levels are 0.05 apart, and in U-shaped (red) plot contours are  0.1 apart.   }
		\label{fig:cont}
	\end{figure}

	In Figure \ref{fig:cont} we provide a simple visualization of the average spatial structure of tracked bridges.  We have used Gaussian kernel density estimation \cite{scikit-learn} to compute the probability density of points after rotating the bridge so that its principle axis is horizontal, and then if necessary applying a second rotation by $\pi$  so that the centre of mass of the tracked points has positive horizontal coordinate. The combined transformation illustrates the extent to which the tracked bridge is elongated, and also the extent to which its centre of mass is displaced from the origin (its start and end point). The density estimates in Figure \ref{fig:cont} were obtained by first generating 400 tracked bridges, and then separately transforming each of them, before applying kernel density estimation to the combined point set. We see from Figure \ref{fig:cont} that the kernel which gives higher asphericity creates bridges with a more elongated density, after our transformation has been applied.

	\section{Discussion}

	In this work we have considered how time variations in the density of information collected about the location of a Brownian bridge affects the typical size and shape of the set of tracked locations. Using methods developed in the context of polymer physics \cite{rud86,rud87}, we have derived general expressions which link the tether point asphericity and gyration radius of this set to the tracking strategy, and have provided analytical formulae for these properties for a range of explicitly defined strategies. Advances in information and communication technology mean tracking data is becoming increasingly available, and the number of applications is likely to grow. For this reason, analytical characterisation of the geometrical properties of tracked paths or, equivalently, of the time decaying trail left by a walker, may be of use in applications. For example, our results might be used to estimate the sizes and shapes of spatial regions contaminated by a diseased walker, as described in the introduction.  When automatically tracking people or animals using surveillance images \cite{del14}, a key problem is the maintenance of identity when the view of the target is interrupted or difficult to resolve due to increased numbers of other individuals. When such effects vary in time, for example due to predictable changes in the number of active individuals, then the average radius of gyration of tracked points may be used to estimate the typical range of the underlying walks. One might also compare empirical radii and asphericities to our results, as a means to verify that the Brownian bridge is a reasonable path model.

\bibliographystyle{unsrt}
\bibliography{MRWRefs}



\section*{Supplementary material}

# SUPPLEMENTARY MATERIAL TO SIZE AND SHAPE OF TRACKED BROWNIAN BRIDGES

ABDULRAHMAN ALSOLAMI, JAMES BURRIDGE AND MICHAŁ GNACIK

The aim of this supplementary material is to provide full details for the proofs of Theorem 1 and Theorem 2 (also included below) from the main document of the article 'Size and Shape of Tracked Brownian Bridges'.

Let $(\widehat{X}(t), \widehat{Y}(t))_{t \in \widetilde{S}}$ be a two-dimensional tracked Brownian bridge with tracking strategy kernel $\mu$ and the rate of intensity $c > 0$. We introduce the following notation

$$M_0(t) = \int_0^t \mu(s)ds, \ \ M_1(t) = \int_0^t M_0(s)ds \text{ and } M_2(t) = \int_0^t M_1(s)ds.$$

**Theorem 1.** *The average radius of gyration of $(\widehat{X}(t), \widehat{Y}(t))_{t \in \widetilde{S}}$, as $c \to \infty$ is*

$$(1) \qquad\qquad r^2 = 2M_1(1) - 4M_2(1).$$

*Proof.* Recall that $r^2 = 2\mathbb{E}\left[\int_0^1 \widehat{X}(t)^2 \mu(t)dt\right]$, so the formula follows from Lemma B.6. $\qquad\square$

**Theorem 2.** *The asphericity of $(\widehat{X}(t), \widehat{Y}(t))_{t \in \widetilde{S}}$, as $c \to \infty$ is*

$$A = 1 - 4\frac{\alpha}{\beta},$$

*where*

$$
\begin{aligned}
(2) \qquad \alpha = & \, 4M_0(1)(M_1(1) - 2M_2(1)) - M_1^2(1) + 4M_1(1)M_2(1) \\
& + 4\int_0^1 t\mu(t)(2M_2(t) - tM_1(t))dt \\
& - 2\int_0^1 tM_0^2(t)dt - 2\int_0^1 (M_1(t) - tM_0(t))^2 dt
\end{aligned}
$$

*and*

$$
\begin{aligned}
(3) \qquad \beta = & -4(M_1^2(1) - 4M_1(1)M_2(1) + 8M_2^2(1)) \\
& + 8\int_0^1 (1-t)tM_0^2(t)dt \\
& + 16\int_0^1 (1-2t)M_0(t)(2M_2(t) - tM_1(t))dt \\
& + 8\int_0^1 M_1(t)((1-4t)M_1(t) + 8M_2(t))dt
\end{aligned}
$$

*Proof.* Recall that $\alpha$ and $\beta$ are given by

$$(4) \qquad \alpha = \left( \mathbb{E}\left[ \int_0^1 \widehat{X}(t)^2 \mu(t) dt \right] \right)^2 - \mathbb{E}\left[ \left( \int_0^1 \widehat{X}(t)\widehat{Y}(t)\mu(t) dt \right)^2 \right]$$

$$(5) \qquad \beta = 2\mathbb{E}\left[ \left( \int_0^1 \widehat{X}(t)^2 \mu(t) dt \right)^2 \right] + 2\left( \mathbb{E}\left[ \int_0^1 \widehat{X}(t)^2 \mu(t) dt \right] \right)^2.$$

We calculate the expectations

$$\mathbb{E}\left[ \int_0^1 \widehat{X}(t)^2 \mu(t) dt \right], \ \mathbb{E}\left[ \left( \int_0^1 \widehat{X}(t)^2 \mu(t) dt \right)^2 \right] \text{ and } \mathbb{E}\left[ \left( \int_0^1 \widehat{X}(t)\widehat{Y}(t)\mu(t) dt \right)^2 \right]$$

in Lemma B.6, Lemma B.9 and Lemma B.13. Then we simplify to obtain the formulae (2) and (3) for $\alpha$ and $\beta$, respectively. □

## 2. Series of lemmas

For the proof of Theorems 1 and 2 in our paper we present a series of lemmas that enables us to find the moments of the elements of the gyration tensor $T$; these lemmas use standard techniques of Itô stochastic calculus [1].

Let $\mu$ be a tracking strategy kernel and let $N \in \mathbb{N}$. We consider the sequence $(M_k)_{k=0}^N$ of functions, defined recursively, as follows

$$(6) \qquad M_0(t) = \int_0^t \mu(s) ds,$$

$$(7) \qquad M_k(t) = \int_0^t M_{k-1}(s) ds$$

for any $k \in \{1, \dots, N\}$. The following notation would also be advantageous $M_{-1}(t) := \mu(t)$ for all $t \in [0,1]$.

**Lemma B.1.** *For $\mu$ and $(M_k)$ we have*

$$\int_0^t s^n \mu(s) ds = \sum_{k=0}^n (-1)^k \frac{n!}{(n-k)!} t^{n-k} M_k(t),$$

$$\int_0^t s^n M_j(s) ds = \sum_{k=0}^n (-1)^k \frac{n!}{(n-k)!} t^{n-k} M_{j+k+1}(t)$$

*for any $j \in \{0, \dots, N-2\}$, $n \in \{0, 1, \dots, N-j-1\}$.*

*Proof.* It follows by simple induction and integration by parts. □

**Lemma B.2.** *Given Brownian motion $B$, denote its natural filtration by $\mathcal{F} = (\mathcal{F}_t)_{t \geq 0}$. Let $X = (X_t)_{t \geq 0}$, $Y = (Y_t)_{t \geq 0}$ and $Z = (Z_t)_{t \geq 0}$ be $\mathcal{F}$-adapted stochastic processes. Then*

$$\mathbb{E}\left[\int_0^t X(s)dB(s)\int_0^t Y(s)dB(s)\int_0^t Z(s)dB(s)\right] = \int_0^t \mathbb{E}\left[Y(s)Z(s)\int_0^s X(r)dB(r)\right]ds$$
$$+ \int_0^t \mathbb{E}\left[X(s)Z(s)\int_0^s Y(r)dB(r)\right]ds$$
$$+ \mathbb{E}\left[\int_0^t X(s)Y(s)ds\int_0^t Z(s)dB(s)\right] \quad .$$

*In particular, if $X$ and $Y$ are deterministic then*

$$\mathbb{E}\left[\int_0^t X(s)dB(s)\int_0^t Y(s)dB(s)\int_0^t Z(s)dB(s)\right] = \int_0^t Y(s)\mathbb{E}\left[Z(s)\int_0^s X(r)dB(r)\right]ds$$
$$+ \int_0^t X(s)\mathbb{E}\left[Z(s)\int_0^s Y(r)dB(r)\right]ds \quad .$$

*Proof.* Follows from Ito's product formula. □

**Lemma B.3.** *Let $B$ be a standard Brownian motion then*

$$B^2(t) = t + 2\int_0^t B(s)dB(s),$$
$$B^3(t) = 3tB(t) + 3\int_0^t (B^2(s) - s)dB(s)$$

*for any $t \geq 0$.*

*Proof.* Simple application of Ito's lemma. □

**Lemma B.4.** *Let $B$ be a standard Brownian motion then*

$$(8) \qquad \int_0^1 M_{k+1}(t)B(t)dt = M_{k+2}(1)B(1) - \int_0^1 M_{k+2}(t)dB(t),$$

$$(9) \qquad \int_0^1 tM_k(t)B(t)dt = (M_{k+1}(1) - M_{k+2}(1))B(1) + \int_0^1 (M_{k+2}(t) - tM_{k+1}(t))\,dB(t),$$

$$(10) \qquad \int_0^1 M_k(t)B^2(t)dt = M_{k+1}(1)B^2(1) - M_{k+2}(1) - 2\int_0^1 M_{k+1}(t)B(t)dB(t)$$

*for $k \in \{-1, 0, \ldots, N-2\}$.*

*Proof.* Note that (8) is a simple consequence of stochastic integration by parts. In order to obtain (9) use Ito's lemma for $f(t, B(t)) = tM_{k+1}(t)B(t)$ then apply (8) and re-arrange. For (10) use Ito's lemma with $f(t, B(t)) = M_{k+1}(t)B^2(t)$. □

**Definition 1.** Let $B$ be a standard Brownian motion and let $\widehat{B}(t) = B(t) - tB(1)$, for all $t \in [0, 1]$. Given such a pair, that is, $(B, \widehat{B})$, we will refer to $\widehat{B}$ as the corresponding Brownian bridge. Let $\widehat{B}$ be a Brownian Bridge and let $B(t) = \widehat{B}(t) + tZ$, where $Z \sim \mathcal{N}(0, 1)$ is independent from $\widehat{B}$. Given such a pair, that is, $(\widehat{B}, B)$, we will refer to $B$ as the corresponding Brownian motion.

Note that $B$ in this pair, that is, $(\widehat{B}, B)$ is not unique, as $Z$ is an arbitrary standard normal random variable independent from $\widehat{B}$.

**Lemma B.5.** *Let* $(\widehat{B}, B)$ *be a pair so that* $\widehat{B}$ *is a Brownian bridge and* $B$ *is the corresponding Brownian motion then*

$$\text{(11)} \qquad \int_0^1 \widehat{B}(t)M_k(t)dt = \int_0^1 B(t)M_k(t)dt + B(1)(M_{k+2}(1) - M_{k+1}(1))$$

$$\text{(12)} \qquad \int_0^1 \widehat{B}^2(t)M_k(t)dt = 2B^2(1)M_{k+3}(1) - M_{k+2}(1) - 2\int_0^1 M_{k+1}(t)B(t)dB(t)$$

$$- 2B(1)\int_0^1 \left(M_{k+2}(t) - tM_{k+1}(t)\right)dB(t).$$

*Proof.* For (11), clearly, $B(t) = \widehat{B}(t) + tZ$, for some $Z \sim \mathcal{N}(0, 1)$ independent from $\widehat{B}$, implies that

$$\text{(13)} \qquad\qquad\qquad \widehat{B}(t) = B(t) - tB(1)$$

Then apply Lemma B.1.

For (12) first express Brownian bridge via (13) and expand the bracket. All the terms are obtained via (9, 10) in Lemma B.4 and Lemma B.1.                                      □

**Lemma B.6.** *Let* $(\widehat{B}, B)$ *be a pair so that* $\widehat{B}$ *is a Brownian bridge and* $B$ *is the corresponding Brownian motion then*

$$\mathbb{E}\left[\int_0^1 \widehat{B}^2(t)M_k(t)dt\right] = M_{k+2}(1) - 2M_{k+3}(1)$$

*for* $k \in \{-1, 0, \dots, N-2\}$.

*Proof.* Use (12) from Lemma B.5 and apply Ito's isometry.                                      □

**Lemma B.7.** *Let* $B$ *be a standard Brownian motion then*

$$\mathbb{E}\left[\left(\int_0^1 B^2(t)M_k(t)dt\right)^2\right] = 3M_{k+1}^2(1) + M_{k+2}^2(1) - 8M_{k+1}(1)\left(\frac{5}{4}M_{k+2}(1) - M_{k+3}(1)\right)$$

$$+ 4\int_0^1 tM_{k+1}^2(t)dt.$$

*for* $k \in \{-1, 0, \dots, N-2\}$.

*Proof.* See Lemma B.2 in [2].                                      □

**Lemma B.8.** *Let $B$ be a standard Brownian motion then*

$$(14) \qquad \mathbb{E}\left[B(1)\int_0^1 tM_k(t)B(t)dt\right] = M_{k+1}(1) - 2M_{k+2}(1) + 2M_{k+3}(1),$$

$$(15) \qquad \mathbb{E}\left[\left(\int_0^1 tM_k(t)B(t)dt\right)^2\right] = (M_{k+1}(1) - M_{k+2}(1))^2$$
$$+ 2(M_{k+1}(1) - M_{k+2}(1))(2M_{k+3}(1) - M_{k+2}(1))$$
$$+ \int_0^1 (M_{k+2}(t) - tM_{k+1}(t))^2 dt,$$

$$(16) \qquad \mathbb{E}\left[\left(B(1)\int_0^1 tM_k(t)B(t)dt\right)^2\right] = 3\left(M_{k+1}(1) - M_{k+2}(1)\right)^2$$
$$+ 6(M_{k+1}(1) - M_{k+2}(1))(2M_{k+3}(1) - M_{k+2}(1))$$
$$+ \int_0^1 \left(M_{k+2}(t) - tM_{k+1}(t)\right)^2 dt$$
$$+ 8\int_0^1 (M_{k+2}(t) - tM_{k+1}(t))M_{k+3}(t)dt$$
$$+ 4\int_0^1 tM_{k+2}(t)(tM_{k+1}(t) - M_{k+2}(t))dt$$

*for $k \in \{-1, 0, \dots, N-2\}$.*

*Proof.* For the first two expectations apply (9) in Lemma B.4, then apply Ito's isometry. For (16) first use (9) in Lemma B.4 to get

$$\left(B(1)\int_0^1 tB(t)M_k(t)dt\right)^2$$
$$= \left((M_{k+1}(1) - M_{k+2}(1))B^2(1) + B(1)\int_0^1 (M_{k+2}(t) - tM_{k+1}(t))dB(t)\right)^2$$
$$= (M_{k+1}(1) - M_{k+2}(1))^4 B^4(1)$$
$$+ B^2(1)\left(\int_0^1 (M_{k+2}(t) - tM_{k+1}(t))dB(t)\right)^2$$
$$+ 2(M_{k+1}(1) - M_{k+2}(1))B^3(1)\int_0^1 (M_{k+2}(t) - tM_{k+1}(t))dB(t).$$

Then apply Lemma B.3 and Lemma B.2. $\qquad\square$

**Lemma B.9.** *Let $(\widehat{B}, B)$ be a pair so that $\widehat{B}$ is a Brownian bridge and $B$ is the corresponding Brownian motion then*

$$\mathbb{E}\left[\left(\int_0^1 \widehat{B}^2(t)M_k(t)dt\right)^2\right] = -3M_{k+2}^2(1) + 12M_{k+2}(1)M_{k+3}(1) - 20M_{k+3}^2(1)$$

$$+ 4\int_0^1 (1-t)tM_{k+1}^2(t)dt$$

$$+ 8\int_0^1 (1-2t)M_{k+1}(t)(2M_{k+3}(t) - tM_{k+2}(t))dt$$

$$+ 4\int_0^1 M_{k+2}(t)((1-4t)M_{k+2}(t) + 8M_{k+3}(t))dt$$

*for $k \in \{-1, 0, \ldots, N-2\}$.*

*Proof.* First note that

$$\left(\int_0^1 \widehat{B}^2(t)M_k(t)dt\right)^2 = \left(\int_0^1 (B(t) - tB(1))M_k(t)dt\right)^2$$

$$= \left(\int_0^1 B^2(t)M_k(t)dt\right)^2 + 4\left(B(1)\int_0^1 tB(t)M_k(t)dt\right)^2$$

$$+ B^4(1)\left(\int_0^1 t^2 M_k(t)dt\right)^2$$

$$- 4B(1)\int_0^1 B^2(t)M_k(t)dt \int_0^1 tB(t)M_k(t)dt$$

$$+ 2B^2(1)\int_0^1 t^2 M_k(t)dt \int_0^1 B^2(t)M_k(t)dt$$

$$- 4B^3(1)\int_0^1 t^2 M_k(t)dt \int_0^1 tB(t)M_k(t)dt.$$

For the expectation of first two terms use Lemmas B.8 and B.7. For the expectation of the fourth term use Lemmas B.4 and Lemma B.2. For the last terms use

$$\mathbb{E}[B^2(1)B^2(t)] = t + 2t^2,$$

$$\mathbb{E}[B^3(1)B(t)] = 3t.$$

After that simplify algebraically. □

**Lemma B.10.** *Let $(\widehat{B}_1, B_1)$, $(\widehat{B}_2, B_2)$ be independent pairs of a Brownian bridge and the corresponding Brownian motion then*

$$\int_0^1 \widehat{B}_1(t)\widehat{B}_2(t)M_k(t)dt = \int_0^1 B_1(t)B_2(t)M_k(t)dt$$

$$- B_1(1)\int_0^1 tB_2(t)M_k(t)dt - B_2(1)\int_0^1 tB_1(t)M_k(t)dt$$

$$+ B_1(1)B_2(1)(M_{k+1}(1) - 2M_{k+2}(1) + 2M_{k+3}(1)).$$

*Proof.* Simple verification. □

**Lemma B.11.** *Let $B_1$ and $B_2$ be two independent standard Brownian motions. Then $\int_0^1 M_k(t)B_1(t)B_2(t)dt$ has the following second moment*

$$\mathbb{E}\left[\left(\int_0^1 B_1(t)B_2(t)M_k(t)dt\right)^2\right] = M_{k+1}^2(1) + 4M_{k+1}(1)(M_{k+3}(1) - M_{k+2}(1)) + 2\int_0^1 tM_{k+1}^2(t)dt$$

*for $k \in \{-1, 0, \ldots, N-2\}$.*

*Proof.* Apply Lemma B.5 in [2] then simplify. □

**Lemma B.12.** *Let $B_1$ and $B_2$ be two independent standard Brownian motions. Then*

$$\mathbb{E}\left[B_1(1)\int_0^1 B_1(t)B_2(t)M_k(t)dt \int_0^1 tB_2(t)M_k(t)dt\right]$$
$$= (M_{k+1}(1) - M_{k+2}(1))(M_{k+1}(1) - 2M_{k+2}(1) + 2M_{k+3}(1))$$
$$+ \int_0^1 tM_k(t)(2M_{k+3}(t) - tM_{k+2}(t))dt$$

*for $k \in \{-1, 0, \ldots, N-2\}$.*

*Proof.* We first apply formula 9 in Lemma B.4 for the term $\int_0^1 tB_2(t)M_k(t)dt$ and obtain

$$\mathbb{E}\left[B_1(1)\int_0^1 B_1(t)B_2(t)M_k(t)dt \int_0^1 tB_2(t)M_k(t)dt\right]$$
$$= (M_{k+1}(1) - M_{k+2}(1))\int_0^1 \mathbb{E}[B_1(1)B_1(t)]^2 M_k(t)dt$$
$$+ \int_0^1 \mathbb{E}[B_1(1)B_1(t)]\mathbb{E}\left[B_2(t)\int_0^1 (M_{k+2}(s) - sM_{k+1}(s))dB_2(s)\right]M_k(t)dt$$

$$= (M_{k+1}(1) - M_{k+2}(1))(M_{k+1}(1) - 2M_{k+2}(1) + 2M_{k+3}(1))$$
$$+ \int_0^1 t\left(\int_0^t (M_{k+2}(s) - sM_{k+1}(s)ds\right)M_k(t)dt$$
$$= (M_{k+1}(1) - M_{k+2}(1))(M_{k+1}(1) - 2M_{k+2}(1) + 2M_{k+3}(1))$$
$$+ \int_0^1 tM_k(t)(2M_{k+3}(t) - tM_{k+2}(t))dt$$

□

**Lemma B.13.** *Let* $(\widehat{B}_1, B_1)$, $(\widehat{B}_2, B_2)$ *be independent pairs of a Brownian bridge and the corresponding Brownian motion then*

$$\mathbb{E}\left[\left(\int_0^1 \widehat{B}_1(t)\widehat{B}_2(t)M_k(t)dt\right)^2\right] = 2M_{k+2}^2(1) + 4M_{k+3}^2(1) - 8M_{k+2}(1)M_{k+3}(1)$$

$$+ 4M_{k+1}(1)(2M_{k+3}(1) - M_{k+2}(1)) + 2\int_0^1 tM_{k+1}^2(t)dt$$

$$+ 2\int_0^1 (M_{k+2}(t) - tM_{k+1}(t))^2 dt$$

$$- 4\int_0^1 tM_k(t)(2M_{k+3}(t) - tM_{k+2}(t))dt$$

*for* $k \in \{-1, 0, \ldots, N-2\}$.

*Proof.* First we use Lemma B.10, independence and symmetry to obtain

$$\mathbb{E}\left[\left(\int_0^1 \widehat{B}_1(t)\widehat{B}_2(t)M_k(t)dt\right)^2\right] = \mathbb{E}\left[\left(\int_0^1 B_1(t)B_2(t)M_k(t)dt\right)^2\right]$$

$$+ 2\mathbb{E}\left[\left(\int_0^1 tB_1(t)M_k(t)dt\right)^2\right]$$

$$+ (M_{k+1}(1) - 2M_{k+2}(1) + 2M_{k+3}(1))^2$$

$$+ 2(M_{k+1}(1) - 2M_{k+2}(1) + 2M_{k+3}(1))^2$$

$$- 4\mathbb{E}\left[B_1(1)\int_0^1 B_1(t)B_2(t)M_k(t)dt\int_0^1 tB_2(t)M_k(t)dt\right]$$

$$+ 2\left(\mathbb{E}\left[B_1(1)\int_0^1 tB_1(t)M_k(t)dt\right]\right)^2$$

$$- 4(M_{k+1}(1) - 2M_{k+2}(1) + 2M_{k+3}(1))^2$$

Term 1 of the above sum is computed via Lemma B.11, term 2 and 6 are calculated via Lemma B.8, to get term 5 use Lemma B.12. After that simplify algebraically. $\qquad\square$

<div align="center">

REFERENCES

</div>

School of Mathematics and Physics, Lion Gate Building, Lion Terrace, Portsmouth, PO1 3HF, UK

*Email address:* michal.gnacik@port.ac.uk, james.burridge@port.ac.uk


\end{document}